\documentclass[a4paper,8pt]{article}
\usepackage[english]{babel}
\usepackage{geometry}
\usepackage[T1]{fontenc}
\usepackage[latin1]{inputenc}
\usepackage{amsmath,amssymb,mathrsfs,amsthm}
\usepackage{soul}
\usepackage{epsfig}
\usepackage{hyperref}
\usepackage{graphicx}
\usepackage{xypic}
\usepackage{datetime}
\usepackage{fancyhdr}
\fancypagestyle{plain}{
\lhead{}
\rhead{}
\lfoot{Date \today{} at \currenttime}

}

\usepackage{tocloft}

\newtheorem{theorem}{Theorem}[section]
\newtheorem{proposition}[theorem]{Proposition}

\newtheorem{lemma}[theorem]{Lemma}

\newtheorem{definition}[theorem]{Definition}%\newenvironment{proof}[1][Démonstration]{\begin{trivlist}
\newcommand{\vc}{\|\cdot\|}
\newcommand{\C}{\mathbb{C}}

\newcommand{\Z}{\mathbb{Z}}
\newcommand{\Si}{\Sigma}

\newcommand{\R}{\mathbb{R}}

\newcommand{\p}{\mathbb{P}}
\newcommand{\eps}{\varepsilon}

\newcommand{\si}{\sigma}
\newcommand{\h}{\mathcal{H}}
\newcommand{\Q}{\mathbb{Q}}
\newcommand{\N}{\mathbb{N}}
\newcommand{\z}{\overline{z}}

\newcommand{\m}{\omega}

\title{On the normalized arithmetic Hilbert function}
\date{\today, \currenttime}
\author{Mounir Hajli}
\begin{document}

\maketitle
\begin{abstract}
Let   $X\subset \p^N_{\overline{\Q}}$ be a  subvariety of
dimension $n$, and 
$\h_{\mathrm{norm}}(X;\cdot)$ 
the normalized arithmetic Hilbert function of $X$ introduced
by Philippon and Sombra. We show
that this function admits the following
asymptotic expansion 
\[
\mathcal{H}_{\rm{norm}}(X;D)=\frac{\widehat{h}(X)}{(n+1)!}D^{n+1}+o(D^{n+1})\quad \forall D\gg 1,
\]
where $\widehat{h}(X)$ is the  normalized height of $X$.
This gives a positive answer to a question raised by  
Philippon and Sombra.
\end{abstract}
\begin{center}

\end{center}

\begin{center}


\end{center}

\section{Introduction}

In \cite{PS}, Philippon and Sombra introduce an arithmetic 
Hilbert function defined for any subvariety in $\p^N$, the
projective space of dimension $N$ over $\overline{\Q}$. This 
function measures the binary complexity of the subvariety. 
In the case of toric subvarieties, a result of Philippon and Sombra shows that 
the asymptotic behaviour of the associated normalized 
 arithmetic Hilbert function is
related to the normalized height of the subvariety considered,
see \cite[Proposition 0.4]{PS}. This result is an important
step toward the proof of the main theorem of \cite{PS}, that
is an explicit formula for the normalized height of
projective translated 
toric varieties, see \cite[Th\'eor\`eme 0.1]{PS}.

In \cite[Question 2.2]{PS}, the authors ask if 
the normalized arithmetic Hilbert function admits an
asymptotic expansion similar to the toric case. More precisely,
given $X$ a subvariety  of dimension $n$ in $\p
^N$ the projective space of dimension $N$ over $\overline{\Q}
$, can we find a real 
$c(X)\geq 0$ such that
\[
\mathcal{H}_{\mathrm{norm}}(X;D)=\frac{c(X)}{(n+1)!}D^{n+1}+
o(D^{n+1})?
\]
If this is the case, do we have 
$c(X)=\widehat{h}(X)$? where $\widehat{h}(X)$ is the
normalized height of $X$.

In this article, we give an affirmative answer to this 
question. We prove the following theorem

\begin{theorem}\label{m7}[Theorem \eqref{m6}]
Let $X\subset \p^N$ be a subvariety of dimension $n$ in 
$\p^N$. Then the normalized arithmetic Hilbert function 
associated to $X$ admits the following asymptotic expansion
\[
\mathcal{H}_{\mathrm{norm}}(X;D)=\frac{\widehat{h}(X)}{(n+1)!}D^{n+1}+
o(D^{n+1}),\quad \forall D\gg 1.
\]

\end{theorem}

The notion of normalized height plays an 
important role in the diophantine approximation on 
tori, in particular in Bogomolov's and generalized 
Lehmer's problems, see \cite{DP}, \cite{Amoroso}. A result
of Zhang shows that a subvariety $X$ with a vanishing
normalized height is necessarily a union of toric subvarieties,
see \cite{Zhang1}.\\

 Gillet and Soul\'e proved an arithmetic 
Hilbert-Samuel formula as a consequence of the arithmetic 
Riemann-Roch theorem, see \cite{ARR}. Roughly speaking, 
 this formula
describes the asymptotic behaviour of the arithmetic
degree of a hermitian module defined by the global sections
of the tensorial power of a positive hermitian line bundle on 
an arithmetic variety. Moreover, the leading term is 
given by the arithmetic degree of the hermitian line bundle.
Later Abb\`es and Bouche gave a new proof for this result
without using the arithmetic Riemann-Roch theorem, see \cite{Abbes}. 
Randriambololona extends the result Gillet and Soul\'e
 to  the case of 
coherent 
sheaf provided as a subquotient of a metrized vector
bundle on  an arithmetic variety, see \cite{Ran}.

\subsection{Notations}

Let $\Q$ be the field of rational numbers, $\Z$ the ring of integers, $K$
a number field and $\mathcal{O}_K$ its ring of integers. For 
$N$ and $D$ two integers in  $\N$ we set 
 $\N_D^{N+1}:=\{a\in \N^{N+1}|\, a_0+\cdots+a_N=D\}$. $\C[x_0,\ldots,x_N]_D$
 (resp. $K[x_0,\ldots,x_N]_D$)
 denotes the complex vector space (resp. $K$-vector space) 
 of homogeneous polynomials of
 degree $D$ in $\C[x_0,\ldots,x_N]$ (resp. in $K[x_0,\ldots,x_N]$).
 
 For any  prime number $p$ we denote by $|\cdot|_p$ the $p$-adic 
 absolute value
  on $\Q$ such that $|p|_p=p^{-1}$ and   by $|\cdot|_\infty$
 or simply $|\cdot|$ the standard absolute value. Let 
 $M_\Q$ be the set of these absolute values. We denote by $M_K$  the set of
 absolute values of $K$ extending the absolute values of $M_\Q$, and
 by $M_K^\infty$ the subset in $M_K$   of 
 archimedean absolute values.
 
We denote by  $\p^N$  the projective space over  $\overline{\Q}$ of dimension $N$. A variety
 is assumed reduced and irreducible.

\subsubsection{Acknowledgements }
I am very grateful to Mart\'in Sombra for his helpful 
conversations and encouragement during the preparation 
of this paper. I would like to thank Vincent Maillot for
his useful discussions.

\section{The proof of Theorem \eqref{m7}}

We keep the same notations as in \cite{PS}. Let $\omega$ be 
the Fubini-Study form on
$\p^N(\C)$. For any $k\in \N_{\geq 1}\cup\{\infty\}$, we denote by $h_k$
the hermitian metric on $\mathcal{O}(1)$ given as follows
\[
h_k(\cdot,\cdot)=\frac{|\cdot|^2}{(|x_0|^{2k}+\cdots+|x_N|^{2k})^{\frac{1}{2k}}},\, \forall k\in \N_{\geq 1}\,\, \text{and}\,\quad h_\infty(\cdot,\cdot)=\frac{|\cdot|^2}{
\max(|x_0|,\ldots,|x_N|)^2}, 
\]
and we let $\overline{\mathcal{O}(1)}_k:=(\mathcal{O}(1),h_k)$ and 
 $\m_k:=c_1(\mathcal{O}(1),h_k)$ for any $k\in \N\cup\{\infty\}$. 
 Note that
 $\m_k=\frac{1}{k}[k]^\ast \m$, where  
$[k]:\p^N(\C)\rightarrow \p^N(\C),\, [x_0:\ldots:x_N]\mapsto [x_0^k:
\ldots:x_N^k]$. Observe that the sequence $(\omega_k)_{k\in \N_{\geq 1}}$ converges weakly to the current $\omega_\infty$. We consider the following
normalized volume form
\[
\Omega_k:= \m_k^{\wedge N}\quad \forall
k\in \N_{\geq 1}\cup \{\infty\}.
\]

For any $k\in \N_{\geq 1}\cup \{\infty\}$, the metric of $\overline{\mathcal{O}(1)}_k$ and 
$\Omega_k$ define a scalar product  $\C[x_0,\ldots,x_N]_D$ denoted by
$\langle \cdot ,\cdot \rangle_k$ given  as
follows 
\begin{equation}\label{m1}
\langle f,g \rangle_k:=\int_{\p^N(\C)}h_k^{\otimes D}(f,g)\Omega_k,
\end{equation}
for any
$f=\sum_a f_a x^a$, $g=\sum_a g_a x^a$ in $\C[x_0,\ldots,x_N]_D$ with
$f_a,g_a\in \C$. We denote by $\vc_k$ the associated 
norm for any $k\in \N_{\geq 1}\cup \{\infty\}$. Note that,
 $\langle f,g \rangle_\infty=\sum_{a}f_a \overline{g}_a$ and 
 $\|x^a\|_\infty=1$ for any $a\in \N_D^{N+1}$ and $D\in \N$. \\

Let $X\subset \p^N$ be a subvariety defined over a number field
$K$. 
Let $v\in M_K^\infty$ and $\si_v:K\rightarrow \C$ the 
corresponding embedding. For any $p_1,\ldots,p_l\in K[x_0,\ldots,
x_N]_D$ we set
\[
\|p_1\wedge \cdots \wedge p_l\|_{k,v}:=\|\si_v(p_1)\wedge \cdots 
\wedge \si_v(p_l)\|_k\quad \forall k\in \N\cup\{\infty\}.
\]

Let $\mathcal{O}(D):=\mathcal{O}(1)^{\otimes D}$. We set 
$M:=\Gamma(\Si,\mathcal{O}
(D)_{|_{\Si}})$ the $\mathcal{O}_K$-module of global sections
of $\mathcal{O}(D)_{|_{\Si}}$, where $\Si$ is the 
Zariski closure of $X$ in $\p^N_{\mathcal{O}_K}$.
For any $v\in M_K^\infty$, we
set $\Gamma(\Si,\mathcal{O}(D)_{|_{\Si}})_{\si_v}:=
\Gamma(\Si,\mathcal{O}(D)_{|_{\Si}})\otimes_{\si_v}\C$. We consider 
the following restriction map 
\[
\pi:\Gamma(\p^N_{\mathcal{O}_K},
\mathcal{O}
(D))_{\si_v}\rightarrow \Gamma(\Si,\mathcal{O}(D)_{|_{\Si}}
)_{\si_v}
\rightarrow 0.
\]
The space $\Gamma(\p^N_{\mathcal{O}_K},
\mathcal{O}
(D))_{\si_v}$ is identified canonically to $K_\si[x_0,\ldots,x_N]_D$.
For any $k\in \N_{\geq 1}\cup\{\infty\}$, this space can
be endowed by the scalar product induced by $\Omega
_k$ and $h_k$, denoted by $\langle \cdot,\cdot \rangle_{k,v}$:
\[
\langle f,g\rangle_{k,v}=\langle \si_v(f),\si_v(g)\rangle_k,
\]
for any $f,g\in \Gamma(\p^N_{\mathcal{O}_K},
\mathcal{O}
(D))_{\si_v}$.
Since $\mathcal{O}(1)$ is ample, then there exists $D_0\in \N$ such that for any $D\geq D_0$,  the restriction map
is surjective. Let $D\geq D_0$, for any $k\in \N\cup \{\infty\}$,
 we denote by $\vc_{k,v,\rm{quot}}$ the quotient norm induced
by $\pi$ and $\vc_{k,v}$.  Following   \cite[p.348]{PS}, we 
endow 
 $\Gamma(\Si,\mathcal{O}(D)_{|_{\Si}})_{\si_v}$
with $\vc_{k,v,\mathrm{quot}}$, for any $k\in \N_{\geq 
 1}\cup
\{\infty\}$.  By this construction, $M$ can be equipped
with a structure of a hermitian $\mathcal{O}_K$-module, denoted
by $\overline{M}_k$. If $f_1,\ldots,f_s\in M$, is a $K$-basis for 
$M\otimes_{\mathcal{O}_K} K$, then
\[
\widehat{\deg}(\overline{M}_k)=\widehat{\deg}(\overline{\Gamma(\Si,\mathcal{O}(D)_{|_{\Si}})}_k
):=\frac{1}{[K:\Q]}\Bigl(\log \mathrm{Card}\bigl(\bigwedge^s M/(f_1\wedge \cdots \wedge f_s)\bigr)  
-\sum_{v:K\rightarrow \C}\log\|f_1\wedge \cdots f_s\|_{k,v} \Bigr).
\]

\subsection{The normalized arithmetic Hilbert function}

Let $X\subset \p^N$ be a subvariety defined over a number field
$K$ and $I:=I(X)\subset K[x_0,\ldots,x_N]$ its ideal of
definition. We set
\[
\h_{geom}(X;D):=\dim_K\bigl(K[x_0,\ldots,x_N]/I \bigr)_D=
\binom{D+N}{N}-\dim_K(I_D).
\]
This function $\h_{geom}(X;\cdot)$ is known as \textit{the 
classical geometric Hilbert function}. In
\cite{PS}, Philippon and Sombra introduce an arithmetic 
analogue of this function. Let $m:=\h_{geom}(X;D)$, 
$l:=\dim_K(I_D)$ and 
\[
\bigwedge^l K[x_0,\ldots,x_N]_D,
\]
the $l$-th exterior power product of $K[x_0,\ldots,x_N]_D$. For
$f\in \bigwedge^l K[x_0,\ldots,x_N]_D $ and $v\in M_K
$ we denote by $|f|_v$ the sup-norm of the coefficients of $f$
at the place $v$, with respect to the standard basis of
$\bigwedge^l K[x_0,\ldots,x_N]_D$. 

\begin{definition} 
(\cite[D\'efinition 2.1]{PS}) Let $p_1,
\ldots,p_l$ be a $K$-basis
of $I_D$, we set
\[
\begin{split}
\mathcal{H}_{\mathrm{norm}}(X;D)=&\sum_{v\in M_K}\frac{[K_v:\Q_v]}{[K:\Q]}
\log|p_1\wedge \cdots \wedge p_m|_v.
\end{split}
\]

\end{definition}

By the product formula, this definition does not depend on
the choice of the basis, also it is invariant by finite
extensions of $K$. 
$\h_{\mathrm{norm}}(X;\cdot)$ is called  
 the normalized arithmetic Hilbert  
function of $X$. 

Following Philippon and Sombra, this arithmetic Hilbert function measures, for any $D\in \N$, the binary
complexity of the $K$-vector space of forms of degree 
$D$ in $K[x_0,\ldots,x_N]$ modulo $I$. As pointed out by
Philippon and Sombra, when $X$ is a toric variety, the
 asymptotic behaviour of its associated normalized arithmetic 
 Hilbert
 function is related
to $\widehat{h}(X)$, the normalized height  of $X$, see \cite[Proposition 0.4]{PS}. The authors ask the following
question:

Given $X$ a subvariety in $\p
^N$ of dimension $n$, can we find a real 
$c(X)\geq 0$ such that
\[
\mathcal{H}_{\mathrm{norm}}(X;D)=\frac{c(X)}{(n+1)!}D^{n+1}+
o(D^{n+1})?
\]
If this is the case, do we have $c(X)=\widehat{h}(X)$?\\

 We recall
the following proposition, which gives a dual formulation for
$\h_{norm}$,
\begin{proposition}\label{w10}
Let $q_1,\ldots,q_m\in K[x_0,\ldots,x_N]_D^{ \vee}$ be a $K$-basis
of $\mathrm{Ann}(I_D)$, then
\[
\begin{split}
\mathcal{H}_{\mathrm{norm}}(X;D)=&\sum_{v\in M_K}\frac{[K_v:\Q_v]}{[K:\Q]}
\log|q_1\wedge \cdots \wedge q_m|_v.
\end{split}
\]
\end{proposition}

\begin{proof}
See \cite[Proposition 2.3]{PS}.
\end{proof}

For any $k\in \N_{\geq 1}\cup \{\infty\}$, we consider the following 
arithmetic function,
\begin{equation}\label{w1}
\begin{split}
\mathcal{H}_{\mathrm{arith}}(X;D,k):=&\sum_{v\in 
M_K^\infty}
\frac{[K_v:\Q_v]}{[K:\Q]}\log \|p_1\wedge \cdots \wedge p_l  \|_{k,v}
\\
&+\sum_{v\in M_K\setminus M_K^\infty}\frac{[K_v:\Q_v]}{[K:\Q]}
\log|p_1\wedge \cdots \wedge p_l|_v+\frac{1}{2}\log(\gamma(N,D,k)),
\end{split}
\end{equation}
where $p_1,\ldots,p_l$ is a $K$-basis of $I_D$ and
\begin{equation}\label{w2}
\gamma(N;D,k):=\prod_{a\in \N_D^{N+1}}\langle a,a\rangle_k^{-1}.
\end{equation}
Notice that for $k=1$, $\mathcal{H}_{\mathrm{arith}}(X;\cdot,1)$
corresponds, up to a constant, to the arithmetic
function $\h_{\mathrm{arith}}(X;\cdot)$ considered in \cite[p. 346]{PS}.\\

Similarly to $\h_{\mathrm{norm}}$, the function 
$\h_{\mathrm{arith}}$ admits a dual formulation.  The scalar product $\langle \cdot,\cdot\rangle_k$ induces the following linear isomorphism
\[
\eta_k:\C[x_0,\ldots,x_N]\rightarrow \C[x_0,\ldots,x_N]^{\vee},\quad
f\mapsto \langle \cdot, f\rangle_k.
\]
Thus $\C[x_0,\ldots,x_N]^{\vee}$ can be endowed with the dual
scalar product, given as follows
\[
\langle \eta_k(f),\eta_k(g)\rangle_k:=\langle f,g\rangle_k, \quad
\forall f,g\in \C[x_0,\ldots,x_N]_D.
\]
We can check easily that, for any $k\in \N\cup \{\infty\}$ we have
 $\|\theta \|_k':=\sup_{g\in \C[x_0,\ldots,x_N]\setminus\{0\}}\frac{|\theta
(g)|}{\|g\|_k}=\|f\|_k$ 
where $f\in \C[x_0,\ldots,x_N]$ is 
such that $\theta=\eta_k(f)$. Then, ${\|\theta\|_k'}^2=\langle \theta,
\theta\rangle_k$ for any $\theta \in \C[x_0,\ldots,x_N]^{\vee}$. 
It follows that,
\begin{equation}\label{w4}
\langle \theta,\zeta\rangle_k=\sum_{b}\langle x^b,x^b\rangle_k^{-1} \theta_b\overline{\zeta}_b.
\end{equation}
This product extends to $\wedge^m\bigl(\C[x_0,\ldots,x_N]_D^{\vee}
\bigr)$ as follows
\[
\langle \theta_1\wedge \cdots \wedge \theta_m,\zeta_1\wedge
\cdots \wedge \zeta_m\rangle_k:=\det(\langle \theta_i,\zeta_j\rangle_k)_{1\leq i,j\leq m}.
\]

\begin{proposition}\label{w5}
Let $q_1,\ldots,q_m\in K[x_0,\ldots,x_N]_D^{ \vee}$ be a $K$-basis
of $\mathrm{Ann}(I_D)$, then
\[
\begin{split}
\mathcal{H}_{\mathrm{arith}}(X;D,k)=&\sum_{v\in M_K^\infty}\frac{[K_v:\Q_v]}{[K:\Q]}\log \|q_1\wedge \cdots \wedge q_m \|_{k,v}^{\vee}\\
&+\sum_{v\in M_K\setminus M_K^\infty}\frac{[K_v:\Q_v]}{[K:\Q]}
\log|q_1\wedge \cdots \wedge q_m|_v.
\end{split}
\]
\end{proposition}
\begin{proof}
The proof is similar to \cite[Proposition 2.5]{PS}.
\end{proof}

\begin{lemma}\label{w6}
There exists $D_1$ such that for any $D\geq D_1$ and any $k\in \N$,
we have
\[
\mathcal{H}_{\mathrm{arith}}(X;D,k)=\widehat{\deg}(\overline{\Gamma(\Si,\mathcal{O}(D)_{|_\Si})}_k)-
\frac{1}{2}\h_{\mathrm{geom}}(X;D)\log \binom{D+N}{N}.
\]

\end{lemma}
\begin{proof}
The proof is similar to \cite[lemme 2.6]{PS}. Let $\mathcal{I}$
be the ideal sheaf of $\Si$ and $\Gamma(\p^n_{\mathcal{O}_K},
\mathcal{I}\mathcal{O}(D))$ the $\mathcal{O}_K$-module of global  
sections of $\mathcal{I}\mathcal{O}(D)$, endowed with the scalar 
products  induced by the scalar product $\langle\cdot,\cdot\rangle_k$. We claim that there exists $D_1$
an integer which does not depend on $k$ such that for any
$D\geq D_1$, we have
\[
\widehat{\deg}(\overline{\Gamma(\Si,\mathcal{O}(D)_{|_{\Si}})}_k)=
\widehat{\deg}(\overline{\Gamma(\p^N_{\mathcal{O}_K},\mathcal{O}
(D))}_k)-\widehat{\deg}(\overline{\Gamma(\p^N_{\mathcal{O}_K},\mathcal{I}\mathcal{O}
(D))}_k).
\]

Indeed,we can find $D_1\in \N$ such that $\forall D\geq D_1$, 
the following  sequence is exact 
\[
0\rightarrow {\Gamma(\p^N_{\mathcal{O}_K},\mathcal{I}
\mathcal{O}
(D)_{|_{\Si}})}\rightarrow {\Gamma(\p^N_{\mathcal{O}_K},
\mathcal{O}
(D))}\rightarrow {\Gamma(\Si,\mathcal{O}(D)_{|_{\Si}})}\rightarrow 0,
\]
and then by \cite[lemme 2.3.6]{Ran1}, the  following sequence of hermitian
 $\mathcal{O}_K$-modules is exact
\[
0\rightarrow \overline{\Gamma(\p^N_{\mathcal{O}_K},\mathcal{I}
\mathcal{O}
(D)_{|_{\Si}})}_k\rightarrow \overline{\Gamma(\p^N_{\mathcal{O}_K},
\mathcal{O}
(D))}_k\rightarrow \overline{\Gamma(\Si,\mathcal{O}(D)_{|_{\Si}})}_k\rightarrow 0,
\]
where the metrics of $\overline{\Gamma(\p^N_{\mathcal{O}_K},
\mathcal{I}\mathcal{O}(D)_{|_{\Si}})}_k$ and
$ \overline{\Gamma(\Si,\mathcal{O}(D)_{|_{\Si}})}_k$
 are induced by the metric of $\overline{\Gamma(\p^N_{\mathcal{O}_K},
\mathcal{O}
(D))}_k$.

We have
\begin{equation}\label{w7}
\widehat{\deg}(\overline{\Gamma(\p^N_{\mathcal{O}_K},
\mathcal{O}
(D))}_k)=\frac{1}{2}\log (\gamma(N;D,k))+\frac{1}{2}\binom{D+N}{N}
\log \binom{N+D}{N}.
\end{equation}
As in the proof of \cite[Lemme 2.6]{PS}, and keeping the same
notations we have,
\begin{equation}\label{w8}
\begin{split}
\widehat{\deg}(&\overline{\Gamma(\Si,\mathcal{O}(D)_{|_{\Si}})}_k)=
\frac{1}{2}\log (\gamma(N;D,k))+\frac{1}{2}\h_{\mathrm{geom}}(X;D)
\log \binom{N+D}{N}\\
&+\sum_{v\in M_K^\infty}\frac{[K_v:\Q_v]}{[K:\Q]}\log \|
p_1\wedge \cdots \wedge p_l \|_{k,v}^{\vee}-\frac{1}{[K:\Q]}
\log \mathrm{Card}\Bigl(\bigwedge^l (I_{\mathcal{O}_K})/
(p_1\wedge \cdots \wedge p_l)   \Bigr).
\end{split}
\end{equation}
The last term in \eqref{w8} does not depend on the metric. It is
computed in
\cite[p. 349]{PS}; we have
\[
\frac{1}{[K:\Q]}
\log \mathrm{Card}\Bigl(\bigwedge^l (I_{\mathcal{O}_K})/
(p_1\wedge \cdots \wedge p_l)   \Bigr)=-\sum_{v\in M_K\setminus M_K^\infty}\frac{[K_v:\Q_v]}{[K:\Q]}
\log|p_1\wedge \cdots \wedge p_l|_v.
\]
This ends the proof of the lemma.
\end{proof}

By \cite[Th\'eor\`eme A]{Ran}, we have
\begin{equation}\label{w9}
\widehat{\deg}(\overline{\Gamma(\Si,\mathcal{O}(D)_{|_{\Si}})}_k)=
\frac{h_{\overline{\mathcal{O}(
1)}_k}(X)}{(n+1)!}D^{n+1}+o(D^{n+1})\quad \forall D\gg 1,
\end{equation}
where $h_{\overline{\mathcal{O}(
1)}_k}(X)$ denotes  the height of the Zariski closure of
$X$ in $\p^N_{\mathcal{O}_K}$ with respect to $\overline{
\mathcal{O}(1)}_k$. Since $\frac{1}{2}\h_{\mathrm{geom}}(X;D)
\log\binom{D+N}{N}=o(D^{n+1})$ for $D\gg 1.$ Then, by Lemma
\eqref{w6}, we get
\begin{equation}\label{s3}
\h_{\mathrm{arith}}(X;D,k)=\frac{h_{\overline{\mathcal{O}(
1)}_k}(X)}{(n+1)!}D^{n+1}+o(D^{n+1})\quad \forall D\gg 1.
\end{equation}

Let $q_1,\ldots,q_m\in K[x_0,\ldots,x_N]^{\vee}$ 
be a $K$-basis of $\mathrm{Ann}(I_D)$.  
For any finite subset  $M$ in $ \N_D^{N+1}$ of cardinal $m$, we set
 $q_M:=(q_{jb})_{1\leq j\leq m, b\in M}\in K^{m\times m}$ where the $q_{jb}$ are such that
 $q_j=\sum_{b\in \N_D^{N+1}}q_{jb}(x^b)^{\vee}$.
 For any $v\in M_K^\infty$,
we have
\begin{equation}
\begin{split}
|q_1\wedge \cdots \wedge q_m|_v&=\max\{|\det(q_M)|_v: M \subset 
\N_D^{N+1}, \mathrm{Card}(M)=m \}\\
&\leq
 \Bigl( 
 \sum_{M; \mathrm{Card}(M)=m}\bigl(
 \prod_{b\in M}\langle b,b\rangle_{v,k}^{-1}\bigr) 
 |\det(q_M)|_v^2\Bigr)^{\frac{1}{2}},
\end{split}
\end{equation}
(We use the following inequality:  $\langle x^a,x^a\rangle_k=\int_{\p^N(\C)}h_{\overline{\mathcal{O}(D)
}_k}(x^a,x^a)\Omega_k\leq 1$ for any $a\in \N_D^{N+1}$, which follows from $h_{\overline{\mathcal{O}(D)
}_k}(x^a,x^a) \leq h_{\overline{\mathcal{O}(D)
}_\infty}(x^a,x^a)\leq 1$ on $\p^N(\C)$, and the 
fact that $\Omega_k$ is 
positive on $\p^n(\C)$ and 
$\int_{\p^N(\C)}\Omega_k=1$).  
 
  Then,
 \begin{equation}
 |q_1\wedge \cdots \wedge q_m|_v\leq \|q_1\wedge \cdots 
 \wedge q_m\|_{k,v}^{\vee}\quad \forall k\in \N.
 \end{equation}
By Propositions \eqref{w10} and \eqref{w5} we get,
\begin{equation}\label{s1}
\h_{\mathrm{norm}}(X;D)\leq \h_{\mathrm{arith}}(X;D,k)\quad\forall k\in \N.
\end{equation}
By \eqref{s3}, the previous inequality gives
\begin{equation}\label{t2}
\limsup_{D\rightarrow \infty}\frac{(n+1)!}{D^{n+1}} \h_{\mathrm{norm}}
(X;D)\leq h_{\overline{\mathcal{O}(1)}_k}(X)\quad \forall k\in \N.
\end{equation}

We know that $(h_k)_{k\in \N}$ converges uniformly to $h_\infty$ 
on $\p^N(\C)$. 
Let $0<\eps<1$, which will be fixed in the sequel, then there exists $k_0\in \N$ such that for 
for any $k\geq k_0$, we have
\[
(1-\eps)^{2D}\leq \frac{(\max(|x_0|_v,\ldots,|x_N|_v))^{2D}}{
(|x_0|_v^{2k}+\cdots+|x_N|_v^{2k})^{\frac{D}{k}}}\leq (1+\eps)^{2D}
\quad \forall x\in \p^N(\C),\forall  D\in \N.
\]
Thus, for any $k\geq k
_0$, $D\in \N_{\geq 1}$ and  $a\in \N_D^{N+1}$ we get
\begin{equation}\label{m3}
\langle x^a,x^a\rangle_k\geq (1-\eps )^{2D}\int_{\p^N(\C)}
h_\infty^{\otimes D}(x^a,x^a)\m_k^N.
\end{equation}

We have 
\begin{align*}
\int_{\p^N(\C)}
h_\infty^{\otimes D}(x^a,x^a)\m_k^N&= \int_{\C^N}\frac{|z^{2a}|}{\max(1,|z_1|,\ldots,|z_N|)^{2D}}\frac{k^N \prod_{i=1}^N|z_i|^{2(k-1)}\prod_{i=1}^Ndz_i\wedge d\z_i}{(1+\sum_{i=1}^N|z_i^{2k}|)^{N+1}}\\
&=2^N\int_{(\R^+)^N}\frac{k^N r^{a+k-1}}{\max(1,r_1,\ldots, r_N)^{2D}}\frac{\prod_{i=1}^Ndr_i}{(1+\sum_{i=1}^N r_i^k)^{N+1}}\\
&=2^N\int_{(\R^+)^N}\frac{r^{\frac{a}{k}}}{\max_{i}(1,r_1,\ldots, r_N)^{\frac{D}{k}}}\frac{\prod_{i=1}^Ndr_i}{(1+\sum_{i=1}^N r_i^k)^{N+1}}\\
&=2^N\sum_{j=0}^N\int_{E_j}\frac{r^{\frac{a}{k}}}{\max_{i}(1,r_1,\ldots, r_N)^{\frac{D}{k}}}\frac{\prod_{i=1}^Ndr_i}{(1+\sum_{i=1}^N r_i)^{N+1}},
\end{align*}
where $E_j:=\{x\in (\R^+)^N|\, x_j\geq 1, x_l\leq x_j\, \text{for}\, l=1,\ldots,N\}$ for $j=1,\ldots,N$  and $E:=\{x\in (\R^+)^N| \,
x_l
\leq 1,\, \text{for}\, l=1,\ldots,N\}$.  Using the following application
\[
(\R^{\ast +})^N\rightarrow (\R^{\ast +})^N,\quad 
x=(x_1,\ldots,x_N)\mapsto (\frac{x_1}{x_j},\ldots,\frac{x_{j-1}}{x_j},\frac{1}{x_j},\ldots,\frac{x_n}{x_j})
\]
for $j=1,\ldots,N$, we can show that there exists $b^{(j)}=
(b_1^{(j)},\ldots,b_N^{(j)})
\in \N^N$ such that
\begin{equation}
\int_{E_j}\frac{r^{\frac{a}{k}}}{\max_{i}(1,r_1,\ldots, r_N)^{\frac{D}{k}}}\frac{\prod_{i=1}^Ndr_i}{
(1+\sum_{i=1}^Nr_i)^{N+1}}
=\int_{E}r^{\frac{b^{(j)}}{k}}\frac{\prod_{i=1}^Ndr_i}{
(1+\sum_{i=1}^Nr_i)^{N+1}}
\end{equation}
We set $b^{(0)}:=a$. Then, 
\begin{equation}\label{m2}
\int_{\p^N(\C)}
h_\infty^D(x^a,x^a)\m_k^N=2^N\sum_{j=0}^N\int_{E}r^{\frac{b^{(j)}}{k}}\frac{\prod_{i=1}^Ndr_i}{
(1+\sum_{i=1}^Nr_i)^{N+1}}
\end{equation}
  
Let $0<\delta<1$, and set $E_\delta:=\{x\in E| x_l\geq \delta\,\, \text{for}\,\,
 l=1,\ldots,N\}$. From \eqref{m3} and \eqref{m2}, we 
 obtain
 \[
 \langle x^a,x^a\rangle_k\geq (1-\eps )^{2D}2^N \sum_{j=0}^N
 \int_{E_\delta} r^{\frac{b^{(j)}}{k}}\frac{\prod_{i=1}^Ndr_i}{
(1+\sum_{i=1}^Nr_i)^{N+1}}\geq (1-\eps )^{2D}2^N (N+1)
\delta^{\frac{D}{k}} \mu_\delta,
 \]
where $\mu_\delta:=\int_{E_\delta}\frac{\prod_{i=1}^Ndr_i}{
(1+\sum_{i=1}^Nr_i)^{N+1}}$.

Thus,
\begin{equation}\label{m5}
\langle x^a,x^a\rangle_k^{-1}\leq 
(1-\eps )^{-2D}
\delta^{-\frac{D}{k}} {\mu_\delta}^{-1}\quad
\forall k\geq k_0, \forall D\in \N_{\geq 1}, \forall a\in 
\N_D^{N+1}.
\end{equation}

Then, for any $k\geq k_0$ and $D\geq D_1,$
\begin{equation}
\begin{split}
\|q_1\wedge \cdots 
 \wedge q_m\|_{k,v}^{\vee}&\leq 
 \Bigl(\sum_{M;\mathrm{Card}(M)=m } \bigl(
 \prod_{b\in M}\langle b,b\rangle_{v,k}^{-1}\bigr)
  \Bigr)^\frac{1}{2} |q_1\wedge \cdots \wedge q_m|_v\\
  &\leq \mathrm{Card}\{M\subset \N_D^{N+1}| \mathrm{Card}
  (M)=m \}^
  \frac{1}{2} (1-\eps)^{-mD}
\delta^{-m\frac{D}{k}} \mu_\delta^{-m}
  |q_1\wedge \cdots \wedge q_m|_v\quad \text{by}\, 
  \eqref{m5}\\
  &\leq \mathrm{Card}(\N_D^{N+1})(1-\eps)^{-mD}\delta^{-m
  \frac{D}{k}} \mu_\delta^{-m}
  |q_1\wedge \cdots \wedge q_m|_v\\
  &=\binom{N+D}{N}^\frac{1}{2}(1-\eps)^{-mD}\delta^{-m
  \frac{D}{k}} \mu_\delta^{-m}
  |q_1\wedge \cdots \wedge q_m|_v.
\end{split}
\end{equation}

Therefore,
\begin{equation}\label{s2}
\begin{split}
\h_{\mathrm{arith}}(X;D,k)\leq & 
\h_{\mathrm{norm}}(X;D)+\frac{1}{2}\log\binom{N+D}{N} -
D\h_{\mathrm{geom}}(X;D)\log (1-\eps)\\
&-
\frac{D\h_{\mathrm{geom}}(X;D) }{k} \log \delta
-\h_{\mathrm{geom}}(X;D) \log \mu_\delta.
\end{split}
\end{equation}

By \eqref{s3}, we obtain that
\begin{equation}\label{t1}
h_{\overline{\mathcal{O}(1)}_k}(X)\leq \liminf_{D\rightarrow 
\infty}
\frac{(n+1)!}{D^{n+1}} \h_{\mathrm{norm}}
(X;D)+O(\eps)+\frac{\log \delta}{k} O(1), \quad \forall k
\geq k_0.
\end{equation}

Gathering  \eqref{t2} and \eqref{t1}, we conclude that 
for any $0<\eps<1$, there exists $k_0\in \N$ such that
\begin{equation}
 \limsup_{D\rightarrow \infty}
\frac{(n+1)!}{D^{n+1}} \h_{\mathrm{norm}}
(X;D)\leq h_{\overline{\mathcal{O}(1)}_k}(X)\leq \liminf_{D
\rightarrow \infty}
\frac{(n+1)!}{D^{n+1}} \h_{\mathrm{norm}}
(X;D)+O(\eps)+\frac{\log \delta}{k} O(1), \quad \forall k
\geq k_0.
\end{equation}

Since
$\lim_{k\rightarrow \infty}h_{\overline{\mathcal{O}(1)}_k}(X)=h_{\overline{\mathcal{O}(1)}_\infty}(X)$ (see for instance 
\cite{Zhang}) and 
$h_{\overline{\mathcal{O}(1)}_\infty}(X)=\widehat{h}(X)$ 
(see \cite[p. 342]{PS})
we get
\begin{equation}\label{normcan}
\liminf_{D\rightarrow \infty}\frac{(n+1)!}{D^{n+1}} \h_{\mathrm{norm}}
(X;D)=\limsup_{
D\rightarrow \infty}\frac{(n+1)!}{D^{n+1}} \h_{\mathrm{norm}}(X;D)=\widehat{h}(X).
\end{equation}

Thus we proved  the following theorem
\begin{theorem}\label{m6}
Let $X\subset \p^N$ be a subvariety of dimension $n$ in 
$\p^N$. Then the normalized arithmetic Hilbert function 
associated to $X$ admits the following asymptotic expansion
\[
\mathcal{H}_{\mathrm{norm}}(X;D)=\frac{\widehat{h}(X)}{(n+1)!}D^{n+1}+
o(D^{n+1})\quad D\gg 1.
\]

\end{theorem}

\bibliographystyle{plain}

\bibliography{biblio}

\vspace{2cm}

Institute of Mathematics, Academia Sinica, Astronomy-Mathematics Building,
6F, No. 1, Sec. 4, Roosevelt Road, Taipei 10617, TAIWAN\\
\emph{E-mail}:\,\ttfamily{hajli@math.sinica.edu.tw, hajlimounir@gmail.com}

\end{document}